\documentclass[10pt]{article}
\usepackage{amssymb,latexsym,amsmath,epsfig,amsthm}

 \makeatletter

\renewcommand\section{\@startsection {section}{1}{\z@}
{-30pt \@plus -1ex \@minus -.2ex}
{2.3ex \@plus.2ex}
{\normalfont\normalsize\bfseries}}

\renewcommand\subsection{\@startsection{subsection}{2}{\z@}
{-3.25ex\@plus -1ex \@minus -.2ex}
{1.5ex \@plus .2ex}
{\normalfont\normalsize\bfseries}}

\renewcommand{\@seccntformat}[1]{\csname the#1\endcsname. }

\makeatother

\begin{document}

\begin{center}
\uppercase{\bf On infinite multiplicative Sidon sets}
\vskip 20pt
{\bf P\'eter P\'al Pach \footnote{{Department of Computer Science and Information Theory, Budapest University of Technology and Economics, 1117 Budapest, Magyar tud\'osok k\"or\'utja 2., Hungary}\\
{\tt ppp@cs.bme.hu}. This author was supported by the Hungarian Scientific Research Funds (Grant Nr. OTKA PD115978 and OTKA K108947) and the J\'anos Bolyai Research Scholarship of the Hungarian Academy of Sciences.}, Csaba S\'andor \footnote{{Institute of Mathematics, Budapest University of
Technology and Economics, H-1529 B.O. Box, Hungary}\\ {\tt csandor@math.bme.hu}.
This author was supported by the OTKA Grant No. K109789. This paper was
supported by the J\'anos Bolyai Research Scholarship of the Hungarian
Academy of Sciences. }}\\

\bigskip

\end{center}

\noindent

\pagestyle{myheadings}
\thispagestyle{empty}
\baselineskip=12.875pt

\newtheorem{theorem}{Theorem}
\newtheorem{lemma}[theorem]{Lemma}
\newtheorem{problem}{Problem}
\newtheorem{corollary}[theorem]{Corollary}
\newtheorem{example}{Example}[section]
\newtheorem{proposition}[theorem]{Proposition}

\begin{abstract}
We prove that if $A$ is an infinite multiplicative Sidon set, then
$\liminf\limits_{n\to \infty}\frac{|A(n)|-\pi (n)}{\frac{n^{3/4}}{(\log n)^3}}<\infty$ and construct an infinite multiplicative Sidon set satisfying $\liminf\limits_{n\to \infty}\frac{|A(n)|-\pi (n)}{\frac{n^{3/4}}{(\log n)^3}}>0$.
\end{abstract}


\section{Introduction}

Throughout the paper we are going to use the notions $[n]=\{1,2,\dots,n\}$ and $A(n)=A\cap [n]$ for $n\in  \mathbb{Z}^+$,  $A\subseteq \mathbb{Z}^+$.

A set $A$ of positive integers is called a {\it multiplicative Sidon set}, if for every $s$ the equation $xy=s$ has at most one solution (up to ordering) with $x,y\in A$. Let $G(n)$ denote the maximal possible size of a multiplicative Sidon set contained in $[n]$. In \cite{ep38} Erd\H{o}s showed that $\pi(n) + c_1n^{3/4}/(\log n)^{3/2}\leq G(n)\leq \pi(n) + c_2n^{3/4}$ (with some $c_1,c_2>0$).
 31 years later Erd\H{o}s \cite{ep69}  himself improved this upper bound to $\pi(n) + c_2n^{3/4}/(\log n)^{3/2}$.
Hence, in the lower and upper bounds of $G(n)$ not only the main
terms are the same, but the error terms only differ in a constant factor.

A generalization of multiplicative Sidon sets is multiplicative $k$-Sidon sets where we require that the equation $a_1a_2\dots a_k=b_1b_2\dots b_k$ does not have a solution with distinct elements taken from the given set. In \cite{ppp} the maximal possible size of a (multiplicative) $k$-Sidon subset of $[n]$ was determined asymptotically precisely, furthermore, lower- and upper bounds were given on the error term.

A closely related problem of Erd\H{o}s-S\'ark\"ozy-T.~S\'os and Gy\H{o}ri is the following: They examined how many elements of the set $[n]$ can be chosen in such a way that none of the $2k$-element products is a perfect square.
 Note that if a set satisfies this property, then it is a multiplicative $k$-Sidon set, since if the equation $a_1a_2\dots a_k=b_1b_2\dots b_k$ has a solution of
distinct elements, then the product of these $2k$ numbers is a perfect
square. For more details, see \cite{ErdSarSos}, \cite{gyori}, \cite{ppp}.

Another related question of Erd\H{o}s asks for the maximal size of a set of integers not containing $k+1$ different numbers such that $a_0\mid a_1a_2\dots a_k$. This question is connected to the minimal possible size of a multiplicative basis of order $k$. For more details, see \cite{chan}, \cite{cgs}, \cite{ps}.

In this paper the maximal possible asymptotic density of a multiplicative Sidon set is investigated. According to the result of Erd\H{o}s, if $A\subseteq \mathbb{Z}^+$ is a multiplicative Sidon set, then for every $n$ we have $A(n)\leq \pi(n) + c_2n^{3/4}/(\log n)^{3/2}$ and the set of primes is, of course, a multiplicative Sidon set for which $|A(n)|=\pi(n)$ for every $n$.

It is not difficult to construct a multiplicative Sidon set for which
$$\limsup\limits_{n\to \infty}\frac{|A(n)|-\pi(n)}{n^{3/4}/(\log n)^{3/2}}>0,$$
that is, for infinitely many values of $n$ the set $A(n)$ can be ``large''. In this paper our aim is to study how large $|A(n)|-\pi(n)$ can be for {\it all} (sufficiently large) values of $n$. That is, how ``large'' a function $f(n)$ can be, if $\liminf\limits_{n\to \infty}\frac{|A(n)|-\pi(n)}{f(n)}>0$. We are going to show in the following theorems that the ``largest'' (up to a constant factor) $f(n)$ for which this holds is $f(n)=\frac{n^{3/4}}{(\log n)^3}$.

More precisely, the following theorems are going to be proven:
\begin{theorem}\label{IMSliminf}
Let $A$ be an infinite multiplicative Sidon set. If
$$
\limsup _{n\to \infty}\frac{|A(n)|-\pi (n)}{\frac{n^{3/4}}{(\log n)^3}}\geq 73643,
$$
then we have
$$\liminf _{n\to \infty}\frac{|A(n)|-\pi (n)}{\frac{n}{(\log n)^{48}}}<0. $$
\end{theorem}
\noindent
Theorem~\ref{IMSliminf} immediately implies the following corollary:
\begin{corollary}
Let $A$ be an infinite multiplicative Sidon set. Then we have
$$\liminf _{n\to \infty}\frac{|A(n)|-\pi (n)}{\frac{n^{3/4}}{(\log n)^3}}<73643. $$
\end{corollary}
\begin{theorem}\label{IMSconst}
There exists a multiplicative Sidon set $A\subseteq \mathbb{N}$ such that
$$\liminf _{n\to \infty}\frac{|A(n)|-\pi (n)}{\frac{n^{3/4}}{(\log n)^3}}>\frac{1}{196608}.$$
\end{theorem}

\section{Proofs}

\smallskip
\noindent
{\bf Proof of Theorem \ref{IMSliminf}}.
Let $A\subseteq \mathbb{Z}^+$ be an infinite multiplicative Sidon set. Throughout this proof $p$ and $p_i$ denote prime numbers. The characteristic function $\chi_{A,n}$ is defined as
$$
\chi_{A,n}(p)= \left\{ \begin{array}{ll}
1, & \mbox{if there exists an $a\in A(n)$ such that $p|a$}\\
0, & \mbox{if $p\nmid a$ for every $a\in A(n)$}

\end{array}
\right.
$$
Erd\H os $\cite{ep38}$ proved that every $m\le n$ may be written in the form $a=uv$, where $v\le u$ and $u\le n^{2/3}$ or $u$ is a prime number.

The following subsets of $A(n)$ play a crucial role in the proof. For every $l\geq 0$ let
$$
A_l^{*}(n)=\left\{ a:a\in A(n)\mbox{ and there exist $u,v$ such that }a=uv,v\le u, n^{1/3}\le v\le \frac{n^{1/2}}{(\log n)^l}\right\}
$$
and
\begin{multline*}
A^{**}(n)=\{ a:a\in A(n)\mbox{ and there exist $u,v$ such that }a=uv,v\le u, v\le n^{1/3},\\ u\le n^{2/3}\mbox{ or $u$ is a prime number}\}
\end{multline*}
First we give upper bounds for $|A_l^{*}(n)|$ and $|A^{**}(n)|$, respectively.

Namely, we are going to prove that
\begin{equation}\label{star}
|A_l^{*}(n)|\le \frac{10n^{3/4}}{(\log n)^{l/2}},
\end{equation}
if $n$ is large enough (depending on $l$). Note that we are going to use this estimation in two cases: $l=0$ and $l=6$.

In order to prove \eqref{star} we are going to use Lemma 2. of \cite{ep69}:
\begin{lemma}\label{Erdos}
Let $G=(V,E)$ be a graph having $t_1$ vertices $x_1,\dots ,x_{t_1}$. Assume that each edge of $G$ is incident to one of the vertices $x_i$, $1\le i\le t_2<t_1$, and that $G$ contains no rectangle (i.e. no circuit of four edges, the rectangle will be denoted by $C_4$). Then
$$
|E|\le t_1+t_1\left[ \frac{t_2}{t_1^{1/2}}\right]+t_2^2\left( 1+\left[ \frac{t_2}{t_1^{1/2}} \right] \right)^{-1}\le t_1+2t_1^{1/2}t_2.
$$
\end{lemma}

Let $L=l\log _2\log n$. According to the definition of $A_l^{*}(n)$, every $a\in A_l^{*}(n)$ can be expressed as $a=uv$ where $v\le u$ and $n^{1/3}\le v\le \frac{n^{1/2}}{(\log n)^l}$. This representation might not be uniquely determined, let us choose for every $a\in A_l^{*}(n)$ the decomposition where $v$ is minimal. As $n^{1/3}\leq v\leq \frac{n^{1/2}}{(\log n)^l}$, there is a unique integer $r\in\left[0,\frac{1}{6}\log_2 n\right]$ such that
\begin{equation}
\label{rL}
\frac{n^{1/2}}{2^{r+L+1}}<v\le \frac{n^{1/2}}{2^{r+L}}.
\end{equation}
Let us take an $r\in\left[0,\frac{1}{6}\log_2 n\right]$ and pick those elements $a\in A(n)$ for which the chosen decomposition $a=uv$ satisfies \eqref{rL} with this choice for $r$.

In this case we have $u\le n^{1/2}2^{r+L+1}\le 2n^{2/3}$. Define the graph $G_r=(V_r,E_r)$ as follows: The vertices are $1,2,\dots, \lfloor n^{1/2}2^{r+L+1}\rfloor$. There is an edge between $u$ and $v$, if $a=uv$ is the chosen representation for some $a\in A_l^{*}(n)$ satisfying \eqref{rL}.

The graph $G_r$ is $C_4$-free, otherwise for some $v_1,v_2,u_1,u_2\in V$ we would have $(v_1,u_1),(v_1,u_2),(v_2,u_1),(v_2,u_2)\in E$. This would imply that $v_1u_1,v_1u_2,v_2u_1,v_2u_2\in A$,  but $(u_1v_1)(u_2v_2)=(u_1v_2)(u_2v_1)$ contradicts the multiplicative Sidon property.

Clearly, $G_r$ satisfies the conditions of Lemma~\ref{Erdos} with $t_2= \frac{n^{1/2}}{2^{r+L}}$ and $t_1= n^{1/2}2^{r+L+1}$. This yields that the number of the edges in graph $G_r$ is at most
$$
|E_r|\le t_1+2t_1^{1/2}t_2\le 2n^{2/3}+2\cdot n^{1/4}2^{\frac{r}{2}+\frac{L}{2}+\frac{1}{2}}\cdot\frac{n^{1/2}}{2^{r+L}}=2n^{2/3}+\sqrt{8}\cdot\frac{n^{3/4}}{2^{\frac{r}{2}+\frac{L}{2}}}.
$$
The number of those $a\in A(n)$ for which $a=uv$ with $v=u$ is at most $n^{1/2}$, therefore
\begin{multline*}
|A_l^{*}(n)|\le n^{1/2}+\sum _{0\le r\le \frac{1}{6}\log _2n}|E_r|\le \\
\leq n^{1/2}+ 2\left(\frac{1}{6}\log _2n +1\right)n^{2/3}+\frac{n^{3/4}}{(\log n)^{l/2}}\sum_{r=0}^{\infty}\frac{\sqrt{8}}{2^{r/2}}\le 10\frac{n^{3/4}}{(\log n)^{l/2}},
\end{multline*}
if $n$ is large enough.

As a next step, we are going to prove that for every positive integer $n$ we have
\begin{equation}\label{starr}
|A^{**}(n)|\le \sum_{n^{2/3}<p\le n}\chi _{A,n}(p)+4n^{2/3}.
\end{equation}
For every $a\in A^{**}(n)$ let us choose the representation $a=uv$, where
\begin{itemize}
\item $v\le u$,
\item $v\le n^{1/3}$,
\item $u\le n^{2/3}$ or $u$ is a prime number
\item and $v$ is minimal.
\end{itemize}
The previous Lemma \ref{Erdos} is applied again. Define the graph $G=(V,E)$ where the vertices are
\begin{itemize}
\item the integers up to $n^{2/3}$,
\item those primes $p$ from the interval $]n^{2/3},n]$ for which there exists an $a\in A(n)$ such that $p|a$
\item  and an extra vertex.
\end{itemize}
There is an edge between $u$ and $v$, if $1\le v< n^{1/3}$; $u\le n^{2/3}$ or $u$ is a prime number; $v<u$ and $a=uv$ is a chosen representation for some $a\in A^{**}(n)$. The graph $G$ is $C_4$-free, otherwise for some $u_1,u_2,v_1,v_2\in V$,
$$(u_1,v_1),(u_2,v_1),(u_1,v_2),(u_2,v_2)\in E$$ we have
$$u_1v_1,u_2v_1,u_1v_2,u_2v_2\in A(n),$$ but
$$(u_1v_1)(u_2v_2)=(u_1v_2)(u_2v_1)$$
 contradicts the multiplicative Sidon property. Thus Lemma~\ref{Erdos} can be applied for $G$ with
$$
t_1=\lfloor n^{2/3}\rfloor +\left(\sum_{n^{2/3}<p\le n}\chi _{A,n}(p)\right) +1,\quad t_2=\lfloor n^{1/3}\rfloor .
$$
In this case we have $\left\lfloor \frac{t_2}{t_1^{1/2}} \right\rfloor =0$. (Note that the extra vertex was added in order to guarantee this.) The number of those $a\in A(n)$ for which $a=uv$ with $v=u$ is at most $n^{1/2}$, therefore
\begin{multline*}
|A^{**}(n)|\le \sqrt{n}+|E|\le \sqrt{n}+\left(\sum_{n^{2/3}<p\le n}\chi _{A,n}(p)\right)+\lfloor n^{2/3}\rfloor +1 +\lfloor n^{1/3}\rfloor ^2 \le \\
\leq\left(\sum_{n^{2/3}<p\le n}\chi _{A,n}(p) \right)+4n^{2/3}.
\end{multline*}
Every $a\in A(n)$ can be written in the form $a=uv$, where $v\le u$ and $u\le n^{2/3}$ or $u$ is a prime number, thus $A(n)\subseteq A_0^{*}(n)\cup A^{**}(n).$ Therefore, by using \eqref{star} and \eqref{starr} we obtain that
\begin{equation}\label{upper}
|A(n)|\le |A_0^{*}(n)|+|A^{**}(n)|\le \sum_{n^{2/3}<p\le n}\chi _{A,n}(p)+11n^{3/4},
\end{equation}
if $n$ is large enough.

According to \eqref{star} and \eqref{starr} we obtain that
$$
|A_6^{*}(n)\cup A^{**}(n)|\le \pi (n)+11\cdot \frac{n^{3/4}}{(\log n)^3}.
$$
Therefore, to prove the theorem it is enough to show that
$$
\liminf _{n\to \infty}\frac{|A(n)\setminus (A_6^{*}(n)\cup A^{**}(n))|}{\frac{n^{3/4}}{(\log n)^3}}<73632.
$$
To prove this bound it suffices to prove the following three statements:
\begin{itemize}
\item Firstly, we are going to show that if $n$ is large enough, then
\begin{equation}\label{containing}
A(n)\setminus (A_6^{*}(n)\cup A^{**}(n)) \subseteq A_1(n)\cup A_2(n),
\end{equation}
where
\begin{multline*}
A_1(n):= \{ a:a\in A(n), a=dp_ip_{i+1}\dots p_s,d\le (\log n)^{12},\\
\frac{n^{1/6}}{(\log n)^6}\le p_i\le p_{i+1}\le \dots \le p_s \le n^{1/2}(\log n)^6\}
\end{multline*}
and
\begin{multline*}
A_2(n):=\{a: a\in A(n), a\geq n/(\log n)^{12}, a=dp_{s-3}p_{s-2}p_{s-1}p_{s},d\le (\log n)^{12}, \\
\frac{n^{1/4}}{(\log n)^9}\le p_{s-3}\le p_{s-2}\le p_{s-1}\le p_s\le n^{1/4}(\log n)^9   \}.
\end{multline*}

\item Secondly, we are going to show that the inequality
$$
\limsup _{n\to \infty}\frac{|A_1(n)|}{\frac{n^{3/4}}{(\log n)^3}}>0 $$
implies
\begin{equation}\label{half}
\liminf _{n\to \infty}\frac{|A(n)|-\pi (n)}{\frac{n^{3/4}}{(\log n)^{48}}}< 0.
\end{equation}

\item Finally, we are going to prove the inequality
\begin{equation}\label{quarter}
\limsup_{n\to \infty}\frac{|A_2(n)|}{\frac{n^{3/4}}{(\log n)^3}}< 73632.
\end{equation}
\end{itemize}

Note that we will refer to these statements by \eqref{containing}, \eqref{half} and \eqref{quarter}. Now, we continue with proving these statements which finishes the proof of Theorem~\ref{IMSliminf}.

\bigskip

To prove statement \eqref{containing}, first let us note that if $a\le \frac{n}{(\log n)^{12}}$, then $a\in A_6^{*}(n)\cup A^{**}(n)$. To see this, let us take the decomposition $a=uv$, where $v\le u$ and $u\le n^{2/3}$ or $u$ is a prime number. The condition $v\le u$ implies $v\le \frac{n^{1/2}}{(\log n)^6}$. Hence,
\begin{itemize}
\item for $n^{1/3}\le v\le \frac{n^{1/2}}{(\log n)^6}$ we have $a\in A_6^{*}(n)$
\item for $v<n^{1/3}$ we have either $u\le n^{2/3}$ or $u$ is a prime number, therefore $a\in A^{**}(n)$.
\end{itemize}
From now on, we are going to assume that $a>\frac{n}{(\log n)^{12}}$.

Let $a=p_1p_2.\dots p_{s}$, where $2\le p_1\le p_2\le \dots \le p_s $ are prime numbers. Five cases are going to be distinguished depending on the size of $p_{s-1}$ and $p_s$.

\medskip
\underline{Case 1} $p_s\ge n^{1/2}(\log n)^6$.\\
The choice $v=\frac{a}{p_s}$, $u=p_s$ shows that $a\in A_6^{*}(n)\cup A^{**}(n)$.

\medskip
\underline{Case 2} There exists a $p_i$ such that $n^{1/3}\le p_i\le \frac{n^{1/2}}{(\log n)^6}$.\\
The choice $v=p_i$ and $u=\frac{a}{p_i}$ shows that $a\in A_6^{*}(n)$.

\medskip
\underline{Case 3} $\frac{n^{1/2}}{(\log n)^6}<p_{s-1}\le p_{s}< n^{1/2}(\log n)^6$.\\
In this case $a=dp_{s-1}p_s$, where $d<(\log n)^{12}$. Hence we have $a\in A_1(n)$.

\medskip
\underline{Case 4} $\frac{n^{1/2}}{(\log n)^6}<p_{s}< n^{1/2}(\log n)^6$ and $p_{s-1}<n^{1/3}$.
\begin{itemize}
\item If  $\displaystyle \prod_{p_l<\frac{n^{1/6}}{(\log n)^6}}p_l>(\log n)^{12}$, then for some $j$ we have $p_1p_2\dots p_{j-1}p_s<n^{1/2}(\log n)^6$ and $p_1p_2\dots p_jp_s\ge n^{1/2}(\log n)^6$, but in this case $p_1p_2\dots p_jp_s\le n^{2/3}$, which implies that for $u=p_1p_2\dots p_jp_s$ and $v=\frac{a}{u}$ we have $a\in A_6^{*}(n)\cup A^{**}(n)$.
\item Otherwise $a=dp_i\dots p_s$, where $\displaystyle d=\prod_{p_l<\frac{n^{1/6}}{(\log n)^6}}p_l \le (\log n)^{12}$ and $\frac{n^{1/6}}{(\log n)^6}\le p_i \le \dots \le p_{s}$. Hence we have $a\in A_1(n)$.
\end{itemize}

\medskip
\underline{Case 5} $p_s<n^{1/3}$.\\
There exists a $k$ such that $p_{k+1}p_{k+2}\dots p_s<n^{1/3}$ but $p_kp_{k+1}\dots p_s\ge n^{1/3}$. Note that $p_kp_{k+1}\dots p_s\le n^{2/3}$, since $p_k\leq p_s<n^{1/3}$.
\begin{itemize}
\item If $n^{1/3}\le p_kp_{k+1}\dots p_s\le \frac{n^{1/2}}{(\log n)^6}$, then $v=p_kp_{k+1}\dots p_s$ and $u=\frac{a}{v}$ shows that $a\in A_6^{*}(n)$.
\item If $n^{1/2}(\log n)^6\le p_kp_{k+1}\dots p_s\le n^{2/3}$, then $u=p_kp_{k+1}\dots p_s$ and $v=\frac{a}{p}$ shows that $a\in A_6^{*}(n)\cup A^{**}(n)$.
\item Finally, let us assume that $\frac{n^{1/2}}{(\log n)^6}<p_kp_{k+1}\dots p_s<n^{1/2}(\log n)^6$. If $\displaystyle \prod_{p_l<\frac{n^{1/6}}{(\log n)^6}}p_l>(\log n)^{12}$, then for some $j$ we have $p_1p_2\dots p_{j-1}p_k\dots p_s<n^{1/2}(\log n)^6$ and $p_1p_2\dots p_jp_k\dots p_s\ge n^{1/2}(\log n)^6$, but in this case $p_1p_2\dots p_jp_k\dots p_s\le n^{2/3}$, thus  $u=p_1p_2\dots p_jp_k\dots p_s$ and $v=\frac{a}{u}$ shows that $a\in A_6^{*}(n)\cup A^{**}(n)$.\\
Therefore, it suffices to prove the statement in the case when $a=dp_i\dots p_s$, where $\displaystyle d=\prod_{p_l<\frac{n^{1/6}}{(\log n)^6}}p_l=p_1p_2\dots p_{i-1} \le (\log n)^{12}$ and $\frac{n^{1/6}}{(\log n)^6}\le p_i \le \dots \le p_{s}<n^{1/3}$. In this case the value of $s-i+1$, that is, the number of the ``large'' prime factors of $a$ can be $3,4,5$ or 6, so $a=dp_{s-2}p_{s-1}p_s$ or $a=dp_{s-3}p_{s-2}p_{s-1}p_s$ or $a=dp_{s-4}p_{s-3}p_{s-2}p_{s-1}p_s$ or $a=dp_{s-5}p_{s-4}p_{s-3}p_{s-2}p_{s-1}p_s$. Now, we are going to check these subcases separately.

\end{itemize}

Subcase 1. $a=dp_{s-2}p_{s-1}p_s$.\\
Let $u=p_{s-2}p_{s-1}$ and $v=dp_s$.
As
$$v=dp_s<n^{1/3}(\log n)^{12}<\frac{n^{1/2}}{(\log n)^6}$$
 and
 $$n^{2/3}>p_{s-2}p_{s-1}=\frac{a}{dp_s}>\frac{n/(\log n)^{12}}{n^{1/3}(\log n)^{12}}=\frac{n^{2/3}}{(\log n)^{24}}>n^{1/2}(\log n)^6,$$
 the decomposition $a=uv$ shows that $a\in A_6^{*}(n)\cup A^{**}(n)$.

\smallskip
Subcase 2. $a=dp_{s-3}p_{s-2}p_{s-1}p_s$.
\begin{itemize}
\item If $p_{s-1}p_s\ge n^{1/2}(\log n)^6$, then for $u=p_{s-1}p_s$ and $v=\frac{a}{u}$ we have
$$p_{s-1}p_s<n^{2/3}$$
and
$$v=\frac{a}{u}\le \frac{n}{n^{1/2}(\log n)^6}=\frac{n^{1/2}}{(\log n)^6},$$
so $a\in A_6^{*}(n)\cup A^{**}(n)$.
\item If $n^{1/4}(\log n)^9<p_s<n^{1/3}$ and $p_{s-1}p_s< n^{1/2}(\log n)^6$, then $p_{s-1}<\frac{n^{1/4}}{(\log n)^3}$, thus $v=p_{s-3}p_{s-2}<\frac{n^{1/2}}{(\log n)^6}$ and $u=dp_{s-1}p_s<(\log n)^{12}n^{1/2}(\log n)^6\le n^{2/3}$ shows that $a\in A_6^{*}(n)\cup A^{**}(n)$.
\item We may assume that $p_s\le n^{1/4}(\log n)^9$.
\begin{itemize}
\item If $p_{s-3}p_{s-2}\le \frac{n^{1/2}}{(\log n)^6}$, then  $u=dp_{s-1}p_s\le (\log n)^{12}(n^{1/4}(\log n)^9)^2\le n^{2/3}$ and $v=p_{s-3}p_{s-2}$ shows that $a\in A_6^{*}(n)\cup A^{**}(n)$.
\item If $p_{s-3}<\frac{n^{1/4}}{(\log n)^9}$ and $p_{s-3}p_{s-2}> \frac{n^{1/2}}{(\log n)^6}$, then $p_{s-2}\ge n^{1/4}(\log n)^3$, therefore $n^{1/2}(\log n)^6\le p_{s-1}p_s\le n^{2/3}$, thus $u=p_{s-1}p_s$ and $v=dp_{s-3}p_{s-2}$ shows that $a\in A_6^{*}(n)\cup A^{**}(n)$.
\item Therefore, we may assume that $a=dp_{s-3}p_{s-2}p_{s-1}p_s$ where
$$\displaystyle d= \prod_{p_l<\frac{n^{1/6}}{(\log n)^6}}p_l=p_1p_2\dots p_{s-4} \le (\log n)^{12}$$
 and
$$\frac{n^{1/4}}{(\log n)^9}\le p_{s-3}\le p_{s-2}\le p_{s-1}\le p_s\le n^{1/4}(\log n)^9,$$
that is, $a\in A_2(n)$.
\end{itemize}
\end{itemize}

\smallskip
Subcase 3. $a=dp_{s-4}p_{s-3}p_{s-2}p_{s-1}p_s$.\\
The inequality
$$n\ge p_{s-4}p_{s-3}p_{s-2}p_{s-1}p_s=\frac{(p_{s-4}p_{s-3}p_s)(p_{s-2}p_{s-1}p_s)}{p_s}>\frac{(p_{s-4}p_{s-3}p_s)^2}{n^{1/3}},$$
yields $p_{s-4}p_{s-3}p_s\le n^{2/3}$. We claim that $dp_{s-2}p_{s-1}\le \frac{n^{1/2}}{(\log n)^6}$.\\
For the sake of contradiction, let us assume that $dp_{s-2}p_{s-1}> \frac{n^{1/2}}{(\log n)^6}$. This would imply
$$\frac{n^{1/2}}{(\log n)^6} <dp_{s-2}p_{s-1}\le (\log n)^{12}p_{s-1}^2,$$
 whence $\frac{n^{1/4}}{(\log n)^9}\le p_{s-1}\le p_s$. Now,
 $$\frac{n^{1/2}}{(\log n)^6}<dp_{s-2}p_{s-1}=\frac{a}{p_{s-4}p_{s-3}p_s}\le \frac{n}{\left (\frac{n^{1/6}}{(\log n)^6}\right)^2\frac{n^{1/4}}{(\log n)^9}}=n^{5/12}(\log n)^{21}$$
 is a contradiction. Hence, $dp_{s-2}p_{s-1}\le \frac{n^{1/2}}{(\log n)^6}$.\\
 The choice $u=p_{s-4}p_{s-3}p_s$ and $v=dp_{s-2}p_{s-1}$ shows that $a\in A_6^{*}(n)\cup A^{**}(n)$.

\smallskip
Subcase 4. $a=dp_{s-5}p_{s-4}p_{s-3}p_{s-2}p_{s-1}p_s$.\\
First of all,
$$n^{1/2}(\log n)^6\le \frac{n^{2/3}}{(\log n)^{24}}\le p_{s-5}p_{s-4}p_{s-3}p_{s-2},$$
thus
$$dp_{s-1}p_s\le \frac{n}{p_{s-5}p_{s-4}p_{s-3}p_{s-2}}\le \frac{n}{n^{1/2}(\log n)^6}=\frac{n^{1/2}}{(\log n)^6}.$$

Also,
$$n\ge p_{s-5}p_{s-4}p_{s-3}p_{s-2}p_{s-1}p_s\ge (p_{s-5}p_{s-4}p_{s-3}p_{s-2})^{3/2},$$
which yields the bound $p_{s-5}p_{s-4}p_{s-3}p_{s-2}\le n^{2/3}$.\\
Hence $u=p_{s-5}p_{s-4}p_{s-3}p_{s-2}$ and $v=dp_{s-1}p_s$ shows that $a\in A_6^{*}(n)\cup A^{**}(n)$. This completes the proof of statement \eqref{containing}.

\bigskip

Now, we continue with proving statement \eqref{half}. We claim that it is enough to prove that for every $c>0$ there exists an $N_0=N_0(c)$ such that for every $n\ge N_0$ and
\begin{multline}\label{lower}
|A_1(n)|=|\{a:a\in A(n),a=dp_i\dots p_s,d\le (\log n)^{12},\\
\frac{n^{1/6}}{(\log n)^6}\le p_i\le \dots \le p_s<n^{1/2}(\log n)^6 \}|>c\cdot\frac{n^{3/4}}{(\log n)^3}
\end{multline}
there exists an $m\in \left[\frac{n^{1/2}}{(\log n)^6}, n^{1/2}(\log n)^6\right]$ such that
\begin{equation}\label{vanM}
\frac{|A(m)|-\pi (m)}{\frac{m}{(\log m)^{48}}}\le -\frac{c^2}{10\cdot 2^{51}+1}.
\end{equation}
First we are going to check that this statement implies statement \eqref{half}, then we are going to prove it.

If the condition of \eqref{half} holds, then there is a $c>0$ and infinite sequence $n_1<n_2<\dots $ such that
\begin{equation*}
|A_1(n_j)|>c\frac{n_j^{3/4}}{(\log n_j)^3}.
\end{equation*}
According to our claim for every large enough $j$  there is an $m_j\in\left[\frac{n_j^{1/2}}{(\log n_j)^8}, n_j^{1/2}(\log n_j)^8\right]$ such that $\frac{|A(m_j)|-\pi (m_j)}{\frac{m_j}{(\log m_j)^{48}}}\le -\frac{c^2}{10\cdot 2^{51}+1}$. Therefore, $\displaystyle \liminf _{n\to \infty}\frac{|A(n)|-\pi (n)}{\frac{n^{3/4}}{(\log n)^{48}}}\le -\frac{c^2}{10\cdot 2^{51}+1}$. Hence, it suffices to prove our claim.

If \eqref{lower} holds, then there exists an integer $d\in\left[1, (\log n)^{12}\right]$ such that
\begin{multline}\label{dfix}
|\{a:a\in A(n),a=dp_i\dots p_s,\frac{n^{1/6}}{(\log n)^6}\le p_i\le \dots \le p_s<n^{1/2}(\log n)^6 \}|>\\
>c\frac{n^{3/4}}{(\log n)^{15}}
\end{multline}
Let us fix such an integer $d$. Let us define a bipartite graph $G=(V,E)$ as follows. Let $V=V_1\cup V_2$, where
 $V_1$ contains the prime number $p$ if there exists an $a\in A(n)$ such that $a=dp_i\dots p_s$ and $\frac{n^{1/6}}{(\log n)^6}\le p_i\le \dots \le p_s=p<n^{1/2}(\log n)^6$ and $V_2$ contains the integers $p_i\dots p_{s-1}$. There is an edge between $v_1\in V_1$ and $v_2\in V_2$ if and only if $dv_1v_2\in A(n)$.

Let $V_2=\{v_1^{(2)},v_2^{(2)},\dots \}$. Let us denote the degree of $v_j^{(2)}$ by $\deg(v_j^{(2)})$. We may assume that $\deg(v_1^{(2)})\ge \deg(v_2^{(2)})\ge \dots $. Let $P$ be the set of prime numbers. Let $P_j\subset P$ such that $p\in P_j$ if and only if the vertex $v_j^{(2)}$ is connected to $p$ in the graph $G$. Clearly, we have $|P_j|=\deg(v_j^{(2)})$.

We claim that $G$ is $C_4$-free. If there is a $C_4$, then there are $p_s,p_{s'}'\in V_1$ and $p_i\dots p_{s-1},p_i'\dots p_{s'-1}'\in V_2$ such that
$$dp_i\dots p_{s-1}p_s,dp_i\dots p_{s-1}p_{s'}',dp_i'\dots p_{s'-1}'p_s,dp_i'\dots p_{s'-1}'p_{s'}'\in A,$$
but
$$((dp_i\dots p_{s-1})p_s)((dp_i'\dots p_{s'-1}')p_{s'}')=((dp_i\dots p_{s-1})p_{s'}')((dp_i'\dots p_{s'-1}')p_s)$$ would contradict the multiplicative Sidon property. Therefore, $G$ is $C_4$-free, so
\begin{equation}\label{C_4}
|P_j\cap P_k|\le 1,\quad \mbox{for $j\ne k$.}
\end{equation}
If $p_s,p_{s}'\in P_j$, then $p_s\not \in A(n)$ or $p_s'\not \in A(n)$ because otherwise
$$(d(p_i\dots p_{s-1})p_s)p_s'=(d(p_i\dots p_{s-1})p_s')p_s$$
 would contradict the multiplicative Sidon property. Hence,
\begin{equation}\label{delete}
|P_j\setminus A(n^{1/2}(\log n)^6)|\ge |P_j|-1.
\end{equation}
Using inequalities \eqref{C_4} and \eqref{delete} we get that
\begin{multline}
|(P_1\cup P_2 \cup \dots \cup P_t)\setminus A(n^{1/2}(\log n)^8)|=\\
=|(P_1\cup (P_2\setminus P_1)\cup (P_3\setminus (P_1\cup P_2))\cup \dots \cup (P_k\setminus (\cup _{j=1}^{k-1}P_j)) \cup \dots \cup (P_t\setminus (\cup _{j=1}^{t-1}P_j)))\setminus A|=\\
=\sum_{k=1}^t|(P_k\setminus (\cup _{j=1}^{k-1}P_j))\setminus A|\ge \\
\ge\sum_{k=1}^t(|(P_k\setminus (\cup _{j=1}^{k-1}P_j))|-1)\ge
\sum_{k=1}^t(|P_k|-(k-1)-1)=\sum_{k=1}^t(|P_k|-k)
\end{multline}
According to \eqref{dfix} and the definition of the graph $G$ we
\begin{equation}\label{logn15}
c\cdot\frac{n^{3/4}}{(\log n)^{15}}\le |E|=\sum _j|\deg(v_j^{(2)})|=\sum_j|P_j|.
\end{equation}
We are going to prove that
$$\deg\left(v^{(2)}_{\left[\frac{c}{2}\cdot \frac{n^{1/4}}{(\log n)^{21}}\right]}\right)\ge \left[\frac{c}{2}\cdot\frac{n^{1/4}}{(\log n)^{21}}\right].$$
 For the sake of contradiction let us suppose that $\deg\left(v^{(2)}_{\left[\frac{c}{2}\cdot\frac{n^{1/4}}{(\log n)^{21}}\right]}\right)< \left[\frac{c}{2}\cdot\frac{n^{1/4}}{(\log n)^{21}}\right]$. Let us split the sum on the right-hand side of \eqref{logn15} into two parts:
\begin{equation}\label{ketsum}
c\cdot\frac{n^{3/4}}{(\log n)^{15}}\le \sum_j\deg(v_j^{(2)})=\sum_{j\le \left[\frac{c}{2}\cdot\frac{n^{1/4}}{(\log n)^{21}}\right]}\deg(v_j^{(2)})+\sum_{j> \left[\frac{c}{2}\cdot\frac{n^{1/4}}{(\log n)^{21}}\right]}\deg(v_j^{(2)}).
\end{equation}
It is well known that $\pi (n^{1/2}(\log n)^6)<\frac{n^{1/2}(\log n)^6}{2}$, if $n$ is large enough, therefore $\deg(v_j^{(2)})\le |V_1|<\frac{n^{1/2}(\log n)^6}{2}$. Hence
\begin{equation}\label{sum1}
\sum_{j\le \left[\frac{c}{2}\cdot \frac{n^{1/4}}{(\log n)^{21}}\right]}\deg(v_j^{(2)})\le \frac{c}{2}\cdot\frac{n^{1/4}}{(\log n)^{21}}\cdot\frac{n^{1/2}(\log n)^6}{2}=\frac{c}{4}\cdot\frac{n^{3/4}}{(\log n)^{15}}.
\end{equation}
Also, $|V_2|\le n^{1/2}(\log n)^6$, since
$p_s\ge \frac{n^{1/2}}{(\log n)^6}$ implies that $p_i\dots p_{s-1}\le n^{1/2}(\log n)^6$. Therefore,
\begin{equation}\label{sum2}
\sum_{j> \left[\frac{c}{2}\cdot\frac{n^{1/4}}{(\log n)^{21}}\right]}\deg(v_j^{(2)})\le \frac{c}{2}\cdot\frac{n^{1/4}}{(\log n)^{21}}\cdot n^{1/2}(\log n)^6=\frac{c}{2}\cdot\frac{n^{3/4}}{(\log n)^{15}}.
\end{equation}
Hence, \eqref{ketsum}, \eqref{sum1} and \eqref{sum2} would imply
$$
c\cdot\frac{n^{3/4}}{(\log n)^{15}}<\frac{3c}{4}\cdot \frac{n^{3/4}}{(\log n)^{15}},
$$
which is a contradiction.

Thus,
$$
\left|(P_1\cup \dots \cup P_{\left[\frac{c}{2}\cdot\frac{n^{1/4}}{(\log n)^{21}}\right]}\setminus A(n^{1/2}(\log n)^8))\right|\ge \sum_{i\le \left[\frac{c}{2}\cdot\frac{n^{1/4}}{(\log n)^{21}}\right]}(\deg(v_i^{(2)})-i)\ge
$$
$$
\left[\frac{c}{2}\cdot\frac{n^{1/4}}{(\log n)^{21}}\right]^2-\binom{\left[\frac{c}{2}\cdot\frac{n^{1/4}}{(\log n)^{21}}\right]+1}{2}>\frac{c^2n^{1/2}}{10(\log n)^{42}},
$$
 if $n$ is large enough.

As
$$\sum\limits_{-6\log _2\log n-1\leq k\leq 6\log _2\log n} \frac{\frac{n^{1/2}}{2^k}}{(\log \frac{n^{1/2}}{2^k})^{48}} <
2^{51} \cdot\frac{n^{1/2}}{(\log n)^{42}}, $$
 there exists an integer $k\in \left[-6\log _2\log n-1, 6\log _2\log n\right]$ such that
 $$
\left|\left( P\left( \frac{n^{1/2}}{2^k}\right)\setminus P\left( \frac{n^{1/2}}{2^{k+1}}\right) \right) \setminus A\left( \frac{n^{1/2}}{2^k}\right)\right|\ge \frac{c^2\frac{n^{1/2}}{2^k}}{10\cdot2^{51}\cdot(\log \frac{n^{1/2}}{2^k})^{48}},
$$
if $n$ is large enough. Let us fix such a $k$.
If $p\in \left( P\left( \frac{n^{1/2}}{2^k}\right)\setminus P\left( \frac{n^{1/2}}{2^{k+1}}\right) \right) \setminus A\left( \frac{n^{1/2}}{2^k}\right)$, then $\chi _{A,\frac{n^{1/2}}{2^k}}(p)=0$, since $p\notin A$ and $2p>\frac{n^{1/2}}{2^k}$. Using \eqref{upper} we get
\begin{multline*}
\left|A\left(\frac{n^{1/2}}{2^k}\right)\right|\le \pi \left(\frac{n^{1/2}}{2^k}\right)-\frac{c^2\frac{n^{1/2}}{2^k}}{10\cdot2^{51}\cdot(\log \frac{n^{1/2}}{2^k})^{48}} +11\left(  \frac{n^{1/2}}{2^k}\right)^{3/4}\le \\
\le \pi \left(\frac{n^{1/2}}{2^k}\right)-  \frac{c^2\frac{n^{1/2}}{2^k}}{(10\cdot2^{51}+1)\cdot(\log \frac{n^{1/2}}{2^k})^{48}},
\end{multline*}
if $n$ is large enough.
The choice $m=\frac{n^{1/2}}{2^k}$ satisfies \eqref{vanM}, thus  statement \eqref{half} holds.

\bigskip

Finally, we prove \eqref{quarter}.
We split into parts the set $A_2(n)$ as follows. Let $a=dp_{s-3}p_{s-2}p_{s-1}p_s\in A_2(n)$ be arbitrary. There exist uniquely determined integers $r$ and $w$ such that
$$\frac{n}{2^{r+1}}<dp_{s-3}p_{s-2}p_{s-1}p_s\le \frac{n}{2^r},$$
$$2^w\le d<2^{w+1}.$$
Since $d\leq (\log n)^{12}$ and $a\geq n/(\log n)^{12}$ we have
$$0\le r\le 12\log _2\log n,$$
$$0\le w\le 12\log _2\log n.$$
Furthermore,
\begin{equation}
\label{rwq1}
\frac{n}{2^{r+w+2}}<p_{s-3}p_{s-2}p_{s-1}p_s\le \frac{n}{2^{r+w}},
\end{equation}
which implies that $p_{s-3}p_{s-2}\le \frac{n^{1/2}}{2^{\frac{r}{2}+\frac{w}{2}}}$. There exists a uniquely determined integer $q$ for which
\begin{equation}
\label{rwq2}
\frac{n^{1/2}}{2^{\frac{r}{2}+\frac{w}{2}+q+1}}<p_{s-3}p_{s-2}\le \frac{n^{1/2}}{2^{\frac{r}{2}+\frac{w}{2}+q}}.
\end{equation}
The lower bound $\frac{n^{1/2}}{(\log n)^{18}}\le p_{s-3}p_{s-2}$ implies
$$0\le q\le 18\log _2\log n.$$
Let
\begin{multline}A_2^{(r,w,q)}(n):=\{ a:a\in A(n), a\geq n/(\log n)^{12} a=dp_{s-3}p_{s-2}p_{s-1}p_s, \\
\frac{n}{2^{r+1}}<dp_{s-3}p_{s-2}p_{s-1}p_s\le \frac{n}{2^r}, 2^w\le d<2^{w+1}, \frac{n^{1/2}}{2^{\frac{r}{2}+\frac{w}{2}+q+1}}<p_{s-3}p_{s-2}\le \frac{n^{1/2}}{2^{\frac{r}{2}+\frac{w}{2}+q}}   \}, \end{multline}
then $A_2(n)$ can be partitioned to the union of the $A_2^{(r,w,q)}(n)$ sets:
$$
A_2(n)=\bigcup\limits _{r=0}^{\lfloor 12\log _2\log n \rfloor} \bigcup\limits _{w=0}^{\lfloor 12\log _2\log n \rfloor} \bigcup\limits _{q=0}^{\lfloor 18\log _2\log n \rfloor} A_2^{(r,w,q)}(n).
$$
We are going to give an upper bound for $|A_2^{(r,w,q)}(n)|$. Let us define the edge-coloured bipartite graph $G_{r,w,q}=(V_{r,w,q},E_{r,w,q})$ as follows. Let $V_{r,w,q}=V_1\cup V_2$, where
\begin{itemize}
\item $V_1$ contains the integers $p_{s-1}p_s$ if and only if there is an $a=dp_{s-3}p_{s-2}p_{s-1}p_s\in A_2^{(r,w,q)}$,
\item $V_2$ contains the integers $p_{s-3}p_{s-2}$ if and only if there is an $a=dp_{s-3}p_{s-2}p_{s-1}p_s\in A_2^{(r,w,q)}$.
\end{itemize}
The vertices $p_{s-1}p_s\in V_1$ and $p_{s-3}p_{s-2}\in V_2$ are connected to each other if and only if there is a $d\in\left[2^w,2^{w+1}\right)$ such that $dp_{s-3}p_{s-2}p_{s-1}p_s\in A_2^{(r,w,q)}(n)$. In this case let the color of this edge be $d$. (Note that there can be more edges between two vertices.) For $v_1\in V_1$ and $2^w\le d<2^{w+1}$ let us denote by $\deg_d(v_1)$ the number of edges of color $d$ starting from $v_1$.

Let us suppose that $p_sp_{s-1},p_{s'}'p_{s'-1}'\in V_1$ and $p_{s-3}p_{s-2},p_{s'-3}'p_{s'-2}'\in V_2$. Then there is no $C_4$ on these points such that
\begin{itemize}
\item edges $(p_{s-1}p_s,p_{s-3}p_{s-2})$ and $(p_{s-1}p_s,p_{s'-3}'p_{s'-2}')$ are of color $d$,
\item edges $(p_{s'-1}'p_{s'}',p_{s-3}p_{s-2})$ and $(p_{s'-1}'p_{s'}',p_{s'-3}'p_{s'-2}')$ are of color $d'$,
\end{itemize}
since otherwise
$$(dp_{s-3}p_{s-2}p_{s-1}p_s)(d'p_{s'-3}'p_{s'-2}'p_{s'-1}'p_{s'}')=(dp_{s'-3}'p_{s'-2}'p_{s-1}p_s)(d'p_{s-3}p_{s-2}p_{s'-1}'p_{s'}')$$
would contradict the multiplicative Sidon property. Hence,
\begin{equation}\label{C4color}
\sum_{v_1\in V_1, 2^w\le d<2^{w+1}}\binom{\deg _d(v_1)}{2}\le \binom{|V_2|}{2}.
\end{equation}
The set of pairs $(v_1,d)$ satisfying $v_1\in V_1$ and $2^w\le d<2^{w+1}$ is split into two classes:
\begin{itemize}
\item the first class contains pairs $(v_1,d)$ if $\deg_d(v_1)\le \left\lfloor \frac{|V_2|}{|V_1|^{1/2}2^{w/2}} \right\rfloor +1$,
\item the second class contains pairs $(v_1,d)$ if $\deg_d(v_1)\ge \left\lfloor \frac{|V_2|}{|V_1|^{1/2}2^{w/2}} \right\rfloor +2$.
\end{itemize}
Clearly,
$$
|A_2^{(r,w,q)}(n)|=\sum _{v_1\in V_1}\sum_{d=2^w}^{2^{w+1}-1}\deg_d(v_1)=\sum _{(v_1,d)\in class_1}\deg_d(v_1)+\sum _{(v_1,d)\in class_2}deg_d(v_1).
$$
The number of pairs $(v_1,d)$ in $class_1$ is at most $|V_1|2^w$, therefore
$$
\sum _{(v_1,d)\in class_1}\deg_d(v_1)\le \left(\left\lfloor \frac{|V_2|}{|V_1|^{1/2}2^{w/2}} \right\rfloor +1\right)|V_1|2^w\le 2^w|V_1|+2^{w/2}|V_1|^{1/2}|V_2|.
$$
By inequality \eqref{C4color} we have
\begin{multline}
\sum _{(v_1,d)\in class_2}\deg_d(v_1)\le \frac{2}{\left\lfloor \frac{|V_2|}{|V_1|^{1/2}2^{w/2}} \right\rfloor +1}\sum_{(v_1,d)\in class_2}\binom{\deg_d(v_1)}{2}\le \\
\leq\frac{2}{\frac{|V_2|}{2^{w/2}|V_1|^{1/2}}}\binom{|V_2|}{2}<2^{w/2}|V_1|^{1/2}|V_2|.
\end{multline}
Hence we obtain that
\begin{equation}
\label{A2becs}
|A_2^{(r,w,q)}(n)|<2^w|V_1|+2\cdot 2^{w/2}|V_1|^{1/2}|V_2|.
\end{equation}
Our aim is to give upper bounds for $|V_1|$ and $|V_2|$, respectively.

Let us start with the upper bound for $|V_2|$: If $p_{s-3}p_{s-2}\in V_2$, then there is a uniquely determined nonnegative integer $t$ such that
\begin{equation}\label{ts-3}
\frac{n^{1/4}}{2^{\frac{r}{4}+\frac{w}{4}+\frac{q}{2}+t+1}}<p_{s-3}\le \frac{n^{1/4}}{2^{\frac{r}{4}+\frac{w}{4}+\frac{q}{2}+t}}.
\end{equation}
According to the definition of $V_2$ we have
$$
\frac{n^{1/4}}{2^{\frac{r}{4}+\frac{w}{4}+\frac{q}{2}-t+1}}<p_{s-2}\le \frac{n^{1/4}}{2^{\frac{r}{4}+\frac{w}{4}+\frac{q}{2}-t-1}}.
$$
We are going to give an upper bound for $t$. As
$$p_{s-1}p_s=\frac{p_{s-3}p_{s-2}p_{s-1}p_s}{p_{s-3}p_{s-2}}\le \frac{\frac{n}{2^{r+w}}}{\frac{n^{1/2}}{2^{\frac{r}{2}+\frac{w}{2}+q+1}}}=\frac{n^{1/2}}{2^{\frac{r}{2}+\frac{w}{2}-q-1}},$$
we get that
$$
\frac{n^{1/4}}{2^{\frac{r}{4}+\frac{w}{4}+\frac{q}{2}-t+1}}<p_{s-2}\le p_{s-1}\le \frac{n^{1/4}}{2^{\frac{r}{4}+\frac{w}{4}-\frac{q}{2}-\frac{1}{2}}},
$$
thus $t\le q+1.5$, that is,
\begin{equation}\label{upper1t}
t\le q+1\le 18\log _2\log n +1.
\end{equation}
Now, with the help of the prime number theorem with error term $\pi(x)=(1+O(\frac{1}{\log x}))\frac{x}{\log x}$ we obtain the following upper bound for those $p_{s-3}p_{s-2}\in V_2$ that satisfy \eqref{ts-3}:
$$
\left( \frac{1}{2}+O\left(\frac{1}{\log n}\right)\right)\frac{n^{1/4}}{2^{\frac{r}{4}+\frac{w}{4}+\frac{q}{2}+t}\log \frac{n^{1/4}}{2^{\frac{r}{4}+\frac{w}{4}+\frac{q}{2}+t}}}  \left( \frac{3}{4}+O\left(\frac{1}{\log n}\right) \right) \frac{n^{1/4}}{2^{\frac{r}{4}+\frac{w}{4}+\frac{q}{2}-t-1}\log \frac{n^{1/4}}{2^{\frac{r}{4}+\frac{w}{4}+\frac{q}{2}-t-1}}}.
$$
Here
$$\log \frac{n^{1/4}}{2^{\frac{r}{4}+\frac{w}{4}+\frac{q}{2}+t}}=\frac{1}{4}\log n+O(\log \log n)$$
and
$$\log \frac{n^{1/4}}{2^{\frac{r}{4}+\frac{w}{4}+\frac{q}{2}-t-1}}=\frac{1}{4}\log n+O(\log \log n).$$
Therefore, the gained upper bound is
$$
\left( 12+O\left(\frac{\log \log n}{\log n}\right)\right) \frac{n^{1/2}}{2^{\frac{r}{2}+\frac{w}{2}+q}(\log n)^2}.
$$
All in all, by using \eqref{upper1t} we get the upper bound
$$
|V_2|\le \left( 12+O\left(\frac{\log \log n}{\log n}\right)\right) \frac{n^{1/2}(q+2)}{2^{\frac{r}{2}+\frac{w}{2}+q}(\log n)^2}.
$$
As a next step, we give an upper bound for $|V_1|$. According to \eqref{rwq1} and \eqref{rwq2} we have
\begin{equation}\label{ss1}
\frac{n^{1/2}}{2^{\frac{r}{2}+\frac{w}{2}-q+2}}\le p_{s-1}p_s\le \frac{n^{1/2}}{2^{\frac{r}{2}+\frac{w}{2}-q-1}},
\end{equation}
therefore $p_{s-1}\le \frac{n^{1/4}}{2^{\frac{r}{4}+\frac{w}{4}-\frac{q}{2}-\frac{1}{2}}}$. There is a uniquely determined integer $t$ for which
 \begin{equation}\label{ss2}
 \frac{n^{1/4}}{2^{\frac{r}{4}+\frac{w}{4}-\frac{q}{2}-\frac{1}{2}+t+1}}<p_{s-1}\le \frac{n^{1/4}}{2^{\frac{r}{4}+\frac{w}{4}-\frac{q}{2}-\frac{1}{2}+t}}.
 \end{equation}
Now, \eqref{ss1} and \eqref{ss2} implies that
$$
\frac{n^{1/4}}{2^{\frac{r}{4}+\frac{w}{4}-\frac{q}{2}+\frac{5}{2}-t}}<p_s<\frac{n^{1/4}}{2^{\frac{r}{4}+\frac{w}{4}-\frac{q}{2}-\frac{3}{2}-t}}
$$
We are going to give an upper bound for $t$. By \eqref{rwq2} and \eqref{ss2} we get that
$$\frac{n^{1/4}}{2^{\frac{r}{4}+\frac{w}{4}+\frac{q}{2}+\frac{1}{2}}}\le p_{s-2}\le p_{s-1}\le \frac{n^{1/4}}{2^{\frac{r}{4}+\frac{w}{4}-\frac{q}{2}-\frac{1}{2}+t}},$$
 which implies
\begin{equation}\label{upper2t}
t\le q+1\le 18\log _2\log n+2.
\end{equation}
Now, with the help of the prime number theorem with error term $\pi(x)=(1+O(\frac{1}{\log x}))\frac{x}{\log x}$ we obtain the following upper bound for those $p_{s-1}p_{s}\in V_1$ that satisfy \eqref{ss2}:
$$
\left( \frac{1}{2}+O\left(\frac{1}{\log n}\right)\right)\frac{n^{1/4}}{2^{\frac{r}{4}+\frac{w}{4}-\frac{q}{2}-\frac{1}{2}+t}\log \frac{n^{1/4}}{2^{\frac{r}{4}+\frac{w}{4}-\frac{q}{2}-\frac{1}{2}+t}}}  \left( \frac{15}{16}+O\left(\frac{1}{\log n}\right) \right) \frac{n^{1/4}}{2^{\frac{r}{4}+\frac{w}{4}-\frac{q}{2}-\frac{3}{2}-t}\log \frac{n^{1/4}}{2^{\frac{r}{4}+\frac{w}{4}-\frac{q}{2}-\frac{3}{2}-t}}}.
$$
Since
$$\log \frac{n^{1/4}}{2^{\frac{r}{4}+\frac{w}{4}-\frac{q}{2}-\frac{1}{2}+t}}=\frac{1}{4}\log n+O(\log \log n)$$
and
$$\log \frac{n^{1/4}}{2^{\frac{r}{4}+\frac{w}{4}-\frac{q}{2}-\frac{3}{2}-t}}=\frac{1}{4}\log n+O(\log \log n),$$
we obtain the upper bound
$$
\left( 30+O\left(\frac{\log \log n}{\log n}\right)\right) \frac{n^{1/2}}{2^{\frac{r}{2}+\frac{w}{2}-q}(\log n)^2}.
$$
By \eqref{upper2t} we get
$$
|V_1|\le \left( 30+O\left(\frac{\log \log n}{\log n}\right)\right) \frac{n^{1/2}(q+2)}{2^{\frac{r}{2}+\frac{w}{2}-q}(\log n)^2}.
$$
Plugging in these bounds for $|V_1|$ and $|V_2|$ in \eqref{A2becs} yields the following upper bound for $|A_2^{(r,w,q)}(n)|$:
\begin{multline*}
|A_2^{(r,w,q)}(n)|\le 2^w|V_1|+2\cdot 2^{w/2}|V_1|^{1/2}|V_2|\le \left( 30+O\left(\frac{\log \log n}{\log n}\right)\right) \frac{n^{1/2}(q+2)}{2^{\frac{r}{2}+\frac{w}{2}-q}(\log n)^2}+ \\
2\left( 30^{1/2}+O\left(\frac{\log \log n}{\log n}\right)\right) \frac{n^{1/4}(q+2)^{1/2}}{2^{\frac{r}{4}+\frac{w}{4}-\frac{q}{2}}\log n} \left( 12+O\left(\frac{\log \log n}{\log n}\right)\right) \frac{n^{1/2}(q+2)}{2^{\frac{r}{2}+\frac{w}{2}+q}(\log n)^2}2^{w/2} =\\
=\left( 24\cdot 30^{1/2}+O\left(\frac{\log \log n}{\log n}\right) \right) \frac{n^{3/4}(q+2)^{3/2}}{2^{\frac{3r}{4}}2^{\frac{w}{4}}2^{\frac{q}{2}}(\log n)^3}.
\end{multline*}
Therefore,
\begin{multline*}
|A_2(n)|\le \sum_{r=0}^{12\log_2\log n+1}\sum_{w=0}^{12\log_2\log n+1}\sum_{q=0}^{18\log_2\log n+1}|A_2^{(r,w,q)}(n)| \le \\
\leq\sum_{r=0}^{\infty}\sum_{w=0}^{\infty}\sum_{q=0}^{\infty}\left( 24\cdot 30^{1/2}+O\left(\frac{\log \log n}{\log n}\right) \right) \frac{n^{3/4}(q+2)^{3/2}}{2^{\frac{3r}{4}}2^{\frac{w}{4}}2^{\frac{q}{2}}{2}(\log n)^3}=\\
=\left( 24\cdot 30^{1/2}+O\left(\frac{\log \log n}{\log n}\right) \right)\left( \sum_{r=0}^{\infty}\frac{1}{2^{\frac{3r}{4}}}\right) \left( \sum_{w=0}^{\infty}\frac{1}{2^{\frac{w}{4}}}\right) \left( \sum_{q=0}^{\infty}\frac{(q+2)^{3/2}}{2^{\frac{q}{2}}}\right) \frac{n^{3/4}}{(\log n)^3}=\\
=\left( 73631.3\dots +O\left(\frac{\log \log n}{\log n}\right)\right) \frac{n^{3/4}}{(\log n)^3},
\end{multline*}
which completes the proof. $\blacksquare$

\bigskip

Now, we continue with the proof of Theorem~\ref{IMSconst}. The following lemma will play an important role in the proof:

\begin{lemma}
\label{4set}
Let $S$ be a set of size $s\geq 56$. Then there exists a family $\mathcal{H}$ of $4$-element subsets of $S$ satisfying the following conditions:
\begin{itemize}
\item[(i)] If $H_1,H_2\in\mathcal{H}$ and $H_1\neq H_2$, then $|H_1\cap H_2|\leq 2$.
\item[(ii)] If $K,L,M,N$ are pairwise disjoint $2$-element subsets of $S$, then at least one of the sets $K\cup L,L\cup M,M\cup N,N\cup K$ does not lie in $\mathcal{H}$.
\item[(iii)] $|\mathcal{H}|\geq s^3/24576$.
\end{itemize}

\end{lemma}

\begin{proof}
Let $p$ be a prime in the interval $(s/8,s/4]$. Note that $p\geq 11$, since $s\geq 56$. It can be supposed that $S\supseteq \mathbb{F}_p\times [4]$. That is, it can be assumed that $S$ contains 4 disjoint copies  of $\mathbb{F}_p$, namely, $A,B,C,D$. We are going to define a family $\mathcal{H}$ of $4$-element subsets such that each element of $\mathcal{H}$ consists of one element from $A$, one from $B$, one from $C$ and one from $D$. For $a,b,c,d\in\mathbb{F}_p$ let $(a,b,c,d)$ denote the $4$-element set $\{(a,1),(b,2),(c,3),(d,4)\}\in S$. We claim that for some $\alpha\in\mathbb{F}_p$, the size of the set
\begin{multline*}
\mathcal{H}_{\alpha}=\{(a,b,c,d)\in\mathbb{F}_p^4:\ a+b+c\ne 0,a+b+d\ne 0, a+c+d\ne 0,b+c+d\ne 0, \\
ab+ac+ad+bc+bd+cd=\alpha                          \}
\end{multline*}
is at least $p^3-4p^2\geq p^3/2$. The size of the set $\{(a,b,c,d)\in\mathbb{F}_p^4:\ a+b+c= 0 \}$ is $p^3$, and the same holds when another triple from $\{a,b,c,d\}$ adds up to 0, therefore,
$$|\{(a,b,c,d)\in\mathbb{F}_p^4:\ a+b+c\ne 0,a+b+d\ne 0, a+b+d\ne 0,b+c+d\ne 0                      \}|\geq p^4-4p^3.$$
There are $p$ possibilities for $\alpha=ab+ac+ad+bc+bd+cd$, which proves that for a well-chosen $\alpha$ we have $|\mathcal{H}_\alpha|\geq p^3/2$. Let us fix such an $\alpha$ and delete some elements of $\mathcal{H}_\alpha$, obtaining $\mathcal{H}$,  in such a way that the multiset $\{a,b,c,d\}$ is different for each element $(a,b,c,d)$ of $\mathcal{H}$. It can be done in such a way that $|\mathcal{H}|\geq |\mathcal{H}_\alpha|/4!$ holds.

\noindent
We claim that $\mathcal{H}$ satisfies the required properties.

Firstly, for checking (i) it is enough to show that the intersection of two elements of $\mathcal{H}$ can not contain exactly 3 elements. Let us assume that $(a,b,c,d_1),(a,b,c,d_2)\in \mathcal{H}$. Then $d_1=\frac{\alpha-(ab+bc+ca)}{a+b+c}=d_2$, so two elements of $\mathcal{H}$ can't differ just in the fourth ``coordinate''. By symmetry, this holds for the first three ``coordinates'', too.

Secondly, for checking (ii) let us assume that
$$(k_1,k_2,l_1,l_2),(m_1,m_2,l_1,l_2),(m_1,m_2,n_1,n_2),(k_1,k_2,n_1,n_2)\in\mathcal{H}.$$
According to the definition of $\mathcal{H}$ the following equations hold:
\begin{equation}
\label{eq1}
k_1k_2+l_1l_2+(k_1+k_2)(l_1+l_2)=\alpha
\end{equation}
\begin{equation}
\label{eq2}
l_1l_2+m_1m_2+(l_1+l_2)(m_1+m_2)=\alpha
\end{equation}
\begin{equation}
\label{eq3}
m_1m_2+n_1n_2+(m_1+m_2)(n_1+n_2)=\alpha
\end{equation}
\begin{equation}
\label{eq4}
n_1n_2+k_1k_2+(n_1+n_2)(k_1+k_2)=\alpha
\end{equation}
Now $\eqref{eq1}-\eqref{eq2}+\eqref{eq3}-\eqref{eq4}$ gives $(k_1+k_2-m_1-m_2)(l_1+l_2-n_1-n_2)=0$. Without the loss of generality it can be assumed that $k_1+k_2=m_1+m_2$. Then $\eqref{eq1}-\eqref{eq2}$ implies that $k_1k_2=m_1m_2$. Thus $\{k_1,k_2\}=\{m_1,m_2\}$. Therefore, $\{k_1,k_2,l_1,l_2\}=\{m_1,m_2,l_1,l_2\}$, so $(k_1,k_2,l_1,l_2)=(m_1,m_2,l_1,l_2)$, hence $K=M$.

Finally, $|\mathcal{H}|\geq |\mathcal{H}_\alpha|/24\geq p^3/48\geq s^3/24576$.

\end{proof}

The following well-known estimations of \cite{Rosser} are going to be used in the proof of Theorem~\ref{IMSconst} to estimate the number of primes up to $x$:
\begin{lemma}\label{primbecs}
For every $x\geq17$ we have $\frac{x}{\log x}< \pi(x)$.
For every $x>1$ we have $\pi(x)\leq 1.26\frac{x}{\log x}$.
\end{lemma}

{\bf Proof of Theorem \ref{IMSconst}}.

Let $P_k$ consist of the primes from the interval $(2^{k-1},2^{k})$. If $k\geq 11$, then by Lemma~\ref{primbecs}
$$|P_k|= \pi(2^k)-\pi(2^{k-1})\geq \frac{2^k}{\log 2^k}-\frac{1.26\cdot 2^{k-1}}{\log 2^{k-1}}\geq \frac{2^k}{4\log 2^k}.$$
Let us apply Lemma~\ref{4set} for $S=P_k$ and let $\mathcal{H}_k$ be the obtained collection of 4-subsets of $P_k$. Let $A_k=\{\prod\limits_{p\in H} p       : H\in \mathcal{H}_k              \}$. Finally, let $A=\{primes\}\cup \bigcup\limits_ {k\geq 11} A_k$.

Now we show that $A$ is a multiplicative Sidon set. Assume that $ab=cd$ for $a,b,c,d\in A$. As each element of $A$ is either a prime or the product of 4 primes, the number of prime factors of $ab$ (counted by multiplicity) is $\Omega(ab)=\Omega(cd)\in \{2,5,8\}$. If $\Omega(ab)=\Omega(cd)=2$, then $\{a,b\}=\{c,d\}$,  and we are done. Now let us assume that $\Omega(ab)=\Omega(cd)=5$. Without the loss of generality it can be assumed that $\Omega(a)=\Omega(c)=4$. Then $\Omega(\gcd (a,c))\geq 3$, therefore $a,c\in A_k$ for some $k$, moreover according to property (i) (of Lemma~\ref{4set}) we get  $a=c$. Then $b=d$ also holds, and we are done. Finally, let us assume that $\Omega(ab)=\Omega(cd)=8$. If $ab$ is not squarefree, that is, divisible by $p^2$ for some prime $p$, then $p$ has to divide $a,b,c,d$, since all elements of $A$ are squarefree. However, it would imply that $\frac{a}{p}|\frac{c}{p}\cdot \frac{d}{p}$, therefore $\Omega(\gcd(a,c))$ or $\Omega(\gcd(a,d))$ would be at least 3. Then, again by property (i) we obtain that $a=c$ (or $a=d$), thus $\{a,b\}=\{c,d\}$. So we can suppose that $ab=cd$ is squarefree. Property (i) and $a|cd$ imply that $\Omega(\gcd(a,c))=\Omega(\gcd(a,d))=2$, so for some primes
$$a=p_1p_2p_3p_4,b=p_5p_6p_7p_8,c=p_1p_2p_5p_6,d=p_3p_4p_7p_8,$$
however this contradicts property (ii) of Lemma~\ref{4set}. Hence, $A$ is a multiplicative Sidon set.

Now we show that for $n\geq 2^{44}$, we have  $|A(n)|\geq \pi(n)+\frac{n^{3/4}}{196608(\log n)^3}$.

If $n\geq 2^{44}$, then $k=\left\lfloor \frac{\log_2 n}{4} \right\rfloor\geq 11$. Therefore, $|P_k|\geq \frac{2^k}{4\log 2^k}>56$, so Lemma~\ref{4set} can be applied for the set $P_k$. Moreover, $|P_k|\geq \frac{2^k}{4\log 2^k}\geq\frac{2^{\frac{\log_2 n}{4}-1}}{4\log 2^{\frac{\log_2 n}{4}-1}}\geq \frac{n^{1/4}}{2\log n}$.
Therefore, $|A(n)|\geq \pi(n)+|A_k|\geq |A_k|+\frac{|P_k|^3}{24576}\geq \pi(n)+\frac{n^{3/4}}{196608(\log n)^3}$. $\blacksquare$

\end{document}